\documentclass[12pt]{amsart}
\usepackage{amsthm,amsmath,amssymb,amscd,graphicx,enumerate, stmaryrd,xspace,verbatim, epic, eepic,url, fullpage}

\usepackage[usenames,dvipsnames]{color}

\usepackage[colorlinks=true]{hyperref}

\usepackage[active]{srcltx}
\usepackage[all]{xypic}
\SelectTips{cm}{}

{
   \newtheorem{theorem}[subsubsection]{Theorem}
      \newtheorem*{theorem*}{Theorem}
   \newtheorem{proposition}[subsubsection]{Proposition}

   \newtheorem{lemma}[subsubsection]{Lemma}
   \newtheorem*{claim}{Claim}
   \newtheorem*{observation}{Observation}

   \newtheorem{problem}[subsubsection]{Problem}
   \newtheorem*{conjecture*}{Conjecture}

}
{\theoremstyle{definition}
   \newtheorem{exercise}[subsubsection]{Exercise}       \newtheorem*{exercise*}{Exercise}
   
   \newtheorem{example}[subsubsection]{Example}
   \newtheorem*{example*}{Example}
   \newtheorem{definition}[subsubsection]{Definition}
   
   \newtheorem*{definition*}{Definition}
   
   \newtheorem{remark}[subsubsection]{Remark}

}
\newcommand{\QQ}{{\mathbb{Q}}}

\newcommand{\ZZ}{{\mathbb{Z}}}

\newcommand{\GG}{{\mathbb{G}}}

\renewcommand{\AA}{{\mathbb{A}}}

\newcommand{\cA}{{\mathcal A}}

\newcommand{\cC}{{\mathcal C}}
\renewcommand{\cD}{{\mathcal D}}

\newcommand{\cF}{{\mathcal F}}
\newcommand{\cG}{{\mathcal G}}
\renewcommand{\cH}{{\mathcal H}}
\newcommand{\cI}{{\mathcal I}}

\newcommand{\cO}{{\mathcal O}}

\newcommand{\oJ}{{\overline{J}}}

\def\<{\langle}
\def\>{\rangle}

\newcommand{\Spec}{\operatorname{Spec}}
\newcommand{\Span}{\operatorname{Span}}

\newcommand{\Proj}{\operatorname{Proj}}

\newcommand{\ox}{{\overline{x}}}

\def\:{{\colon}}
\def\.{{,\dots,}}
\def\dim{{\rm dim}}

\def\inv{{\rm inv}}

\newcommand{\double}{\genfrac..{0pt}1
{\raise -1pt\hbox{$\scriptstyle\longrightarrow$}}{\raise 3pt\hbox
{$\scriptstyle\longrightarrow$}}}

\renewcommand{\setminus}{\smallsetminus}

\def\int{{\rm int}}

\def\tototi{\mathbin{\mathop{\otimes}\limits^{\raise-1pt\hbox
{$\scriptscriptstyle {\rm L}$}}}}

\def\indlim{\mathop{\vrule width0pt height7pt depth
4pt\smash{\lim\limits_{\raise 1pt\hbox to 14.5pt
{\rightarrowfill}}}}}
\def\projlim{\mathop{\vrule width0pt height7pt depth
4pt\smash{\lim\limits_{\raise 1pt\hbox to 14.5pt
{\leftarrowfill}}}}}

\newcommand\displaceamount{3pt}

\newcommand{\doubledown}{\ar@<\displaceamount>[d]\ar@<-\displaceamount>[d]}

\newcommand{\doubleup}{\ar@<\displaceamount>[u]\ar@<-\displaceamount>[u]}

\newcommand{\doubleright}{\ar@<\displaceamount>[r]\ar@<-\displaceamount>[r]}

\newcommand{\ord}{{\operatorname{ord}}}

\def\tilR{{\widetilde R}}

\begin{document}
\title[Resolution and foliations]{Resolution of singularities for the dynamical mathematician}

\author[Abramovich]{Dan Abramovich}
\address{Department of Mathematics, Box 1917, Brown University,
Providence, RI, 02912, U.S.A}
\email{dan\_abramovich@brown.edu}

\date{\today}

\thanks{This is a report on research  supported by BSF grants 2018193, 2022230, ERC Consolidator Grant 770922 - BirNonArchGeom, NSF grants
DMS-2100548, DMS-2401358, the Plan d'Investissements France 2030, IDEX UP ANR-18-IDEX-0001, and Simon foundation grants MPS-SFM-00006274, MPS-TSM-00008103. I thank my collaborators Bellotto da Silva, Temkin and W{\l}odarczyk for this exciting project and their part in securing the above support. I thank IHES, the Hebrew University, CIRM and BNF   for their hospitality during 2024-2025. I thank the participants of the CIRM workshop \emph{Foliations, birational geometry and applications}, 3-7 February, 2025, where many stimulating discussions occurred.}

\begin{abstract}
I begin by  explaining to non-specialists why resolution of singularities in characteristic 0 works. Then I go into some ideas telling how it actually works. I finish with a brief discussion of related results on foliations.
\end{abstract}
\maketitle

\setcounter{tocdepth}{1}
\tableofcontents

\section{Introduction}

These are lecture notes for a minicourse by the same title delivered at the CIRM conference \emph{Foliations, birational geometry and applications,}
3-7 February, 2025

We work in characteristic 0.  The main result is due to Hironaka: resolution of singularities in characteristic 0 exists.

I report on work with Andr\'e Belotto da Silva, Michael Temkin, and Jaros{\l}aw W{\l}odarczyk; any claim to originality is joint with them and appears in the paper \cite{ABTW-foliated}. There is also related work with Ming Hao Quek and Bernd Schober which influenced my presentation but is not explicitly described here.

There is a whole lot of other work on the subject, to some of which I will hopefully  refer at appropriate junctures. Some reference are  \cite{Hironaka,Giraud-etude,Giraud,Villamayor,Bierstone-Milman,Encinas-Hauser,Encinas-Villamayor,Wlodarczyk,Kollar,ATW-weighted}.

\section{Blowing up}

I describe an approach to blowing up that originates with  Rees's $\mathfrak{a}$-transform of a ring $A$ with ideal $\mathfrak a$. It was rediscovered many times, each time with new angles and insights:  Fulton and MacPherson use it under the name \emph{degeneration to the normal cone}, a key construction in intersection theory, a term I stick by \cite{Fulton}. Miles Reid showed its relevance for GIT and the minimal model program, bringing in flips to the picture \cite{Reid}.  For Cox it is a case of the Cox construction \cite{Cox}. W{\l}odarczyk used it first for factorization of birational maps \cite{W-cobordism} and more recently for resolution \cite{Wlodarczyk-cobordant}. Quek and Rydh used it specifically to describe weighted blowups \cite{Quek-Rydh}.

\subsection{Classical blowup in concrete terms --- the cox construction}

Let us blow up the vanishing locus of the ideal $\oJ = (x_1,\ldots,x_k)$ on the variety $Y = \Spec k[x_1,\ldots,x_n]$, the affine space.

Define the \emph{extended Rees algebra} $$\tilR\quad  :=\quad  \cO_Y[s,x_1',\ldots,x_k']\  /\  (x_1-sx_1',\ \ldots,\  x_k-sx_k'),$$ and denote $$B := \Spec_Y \tilR.$$ 

The variable $s$ is the \emph{exceptional variable}, and it has global meaning: the $Y$-variaty $B$ is the \emph{deformation to the normal cone} of $\oJ$, and $s$ is the deformation parameter.  The variables $x_i'$ are the \emph{transformed variables,} associated to the generators $x_i$ of $\oJ$. These depend on the choice of generators.

We denote by $$V:=V(x_1',\ldots,x_k')$$ the \emph{vertex} of $B$.

The algebra $\tilR$ is graded, with $x_i'$ in degree 1 and $s$ in degree $-1$. This corresponds to the natural action of the multiplicative group $\GG_m$ on $B$ relative to $Y$, given by 
$$t\cdot (s,x_1',\ldots,x_k') = (t^{-1} s, \ tx_1',\ldots, tx_k').$$

This action leaves the vertex invariant, providing an action on $B_+ := B \smallsetminus V$. The action on $B_+$ is free, since if $x_i' \neq 0$ and $t\neq 1$ we have $tx_i'\neq x_i'$.

We can therefore consider the quotient 
$$Bl_\oJ(Y) := B_+ / \GG_m,$$
the \emph{blowup} of $Y$ along $\oJ$.  This is in fact a variety, as we recall below.

\subsection{Classical blowup in concrete terms --- old charts}

The construction above, using extended Rees algebras, deformation to the normal cone, and $\GG_m$-actions, is periodically rediscovered every few years, each time shedding new light. I contend that it is the \emph{right} way to introduce blowups, albeit requiring the sophistication of group actions.

Since we did learn another presentation of a blowup in a basic course, let us see that it is the same thing, and at the same time be convinced that $Bl_\oJ(Y)$ is a variety.

The variety $B_+$ is covered by the affine opens $B_i$ where $x_i'\neq 0$. For instance

$$B_1 = \Spec_Y \cO_Y[s,x_1',\ldots,x_k',x_1'^{-1}] / (x_1-sx_1',\ldots, x_k-sx_k').$$

The action of $\GG_m$ on $B_1$ admits a \emph{slice} $Y_1$ where $x_1' = 1$, on which $s=x_1$ and $x_j = x_1x_j'$.  In other words $B_i/ \GG_m \simeq Y_i$ is an affine variety, and the transition from $Y$ to $Y_i$, and between the different charts $Y_i$, is given by the equations we learned in a standard algebraic geometry course. 

 \begin{exercise} Show directly that the invariants of the action are given by $x_i$ and $y_i = x_i'/x_1'$, so that $$B_1 = \Spec_Y\cO_Y[y_2,\ldots,y_k] / (x_2-x_1y_2,\ldots, x_k-x_1y_k)[x_1',x_1'^{-1}] = \Spec_{Y_1} \cO_{Y_1}[x_1',x_1'^{-1}] $$ and the blowup is given by the standard affine charts $$Y_1=\Spec_Y \cO_Y[y_2,\ldots,y_k] / (x_2-x_1y_2,\ldots, x_k-x_1y_k).$$
\end{exercise}
\subsection{Classical blowup in concrete terms --- global interpretation}
If $Y$ is general and $\oJ$ is an ideal, then $$\tilR = \bigoplus_{a\in \ZZ} R_a,$$ where $R_a = \oJ^a$ for $a\geq 0$ and $\cO_Y$ for $a\leq 0$. In other words, just as the blowup is given as the $\Proj$ of the standard Rees algebra of $\oJ$, we have given it here a variant quotient construction in terms of the \emph{extended} Rees algebra. 

The element $s$ corresponds to the generator $1$ of the ideal $\oJ_{-1} = \cO_Y$ in degree $-1$.

Rees actually mainly introduced the extended Rees algebra, and used it to great advantage to study ideals. This continued through history --- many of the deeper properties of blowups are studied through the deformation to the normal cone.

\subsection{Weighted blowup  --- the cox construction}\label{Sec:weighted} I follow Quek and Rydh \cite{Quek-Rydh} and W{\l}odarczyk \cite{Wlodarczyk-cobordant}.

Let us now blow up the object  $\oJ = (x_1^{1/w_1},\ldots,x_k^{1/w_k})$ on the same variety $Y = \Spec k[x_1,\ldots,x_n]$. I will say what $\oJ$ is --- it is no longer an ideal --- but first let us see it in action.

The extended Rees algebra is now $$\tilR \quad := \quad \cO_Y[s,x_1',\ldots,x_k'] \ /\  (x_1-s^{w_1}x_1',\ \ldots,\  x_k-s^{w_k}x_k'),$$ note the powers of $s$ in the ideal, and denote $$B := \Spec_Y \tilR.$$ 

The $Y$-variaty $B$ is now the deformation of $Y$ to the  \emph{weighted} normal cone of $\oJ$.  We still have \emph{transformed variables} $x_i'$ associated to the generators $x_i$ of $\oJ$. 

We still denote by $V:=V(x_1',\ldots,x_k')$ the \emph{vertex} of $B$.

The algebra $\tilR$ is graded, with $s$ in degree $-1$, but now $x_i'$ have degree $w_i$. This corresponds to the natural action of the multiplicative group $\GG_m$ on $B$ relative to $Y$, given by 
$$t\cdot (s,x_1',\ldots,x_k') = (t^{-1} s, \ t^{w_1}x_1',\ldots, t^{w_k}x_k').$$

This action still leaves the vertex invariant, providing an action on $B_+ := B \smallsetminus V$. 

The action on $B_+$ is no longer free, since if $x_i' \neq 0$ the group $\mu_{w_i}$ fixes $x_i'$ and, when $w_i >1$, always has fixed points where the other $x_j'$ vanish.

It is therefore natural to consider the quotient \emph{stack}
$$Bl_\oJ(Y) := [B_+ / \GG_m],$$
the stack-theoretic \emph{weighted blowup} of $Y$ along $\oJ$.  This is a Deligne--Mumford stack.\footnote{In positive characteristic it is only a tame algebraic stack.}

\subsection{What is $\oJ$? What is $\tilR$?} 

For the discussion here, which is mostly local, we can view $\oJ$ as a \emph{monomial valuation} $v_\oJ$, where $v_\oJ(x_i) = w_i$. A global interpretation is that $\oJ$ is a $\QQ$-ideal of a particularly nice shape.

 A $\QQ$-ideal is a notion that in this discussion is equivalent to Hironaka's \emph{idealistic exponent} \cite{Hironaka-idealistic}. A $\QQ$-ideal is an object of the form $\cI^{1/b}$, an equivalence class given by ideal $\cI$ and positive integer $b$, where $\cI_1^{1/b_1}$ is equivalent to $\cI_2^{1/b_2}$ if $\cI_1^{b_2}$ and $\cI_2^{b_1}$ have the same integral closure. One can add $\QQ$-ideals, and $\oJ = (x_1)^{1/w_1} + \cdots + (x_k)^{1/w_k}$.

The extended Rees algebra $\tilR = \bigoplus_{a\in \ZZ} R_a$ is determined by the valuation as follows:

$$R_a \ =\  \{f\in \cO_Y\ :\  v_\oJ(f)\geq a\}.$$

In these terms $B = \Spec_Y\tilR$, the vertex is $V = V(\oplus_{a>0} R_a)$, and $B_+ = B \setminus V$. 

\subsection{Deformation to the normal cone --- the big picture}

First note that  we have a globally defined morphism $B \to \Spec k[s]$. 

Where $s=1$ we have $x_i = x_i'$ so the fiber is just $Y$. A similar picture holds when we fix any $s\neq 0$.

Over $s=0$ we have that $x_i=0$, so the fiber lies over $Z:= V(\oJ) \subset Y$. Over a point of $Z$ the variables $x_i'$ are free to vary arbitrarily, in other words the fiber is an affine-space bundle over $Z$ --- that's the weighted normal cone.\footnote{In the truly weighted case the fiber over $s=0$  is not necessarily a vector bundle, but the action of $\GG_m$ provides enough structure to work with it.}

The action of $\GG_m$ is a product action over $V \simeq Z \times \AA^1$, pulling towards $s=0$. The action on the normal cone expands $x_i'$ with weight $w_i$. It is instructive to draw the picture (as I do in my lectures), and the picture remains important later on. 
\begin{exercise}\label{Ex:picture} Draw the picture! \label{Ex:picture}
\end{exercise}
Don't proceed any further until you either solve the exercise or peek in its solution, see Section \ref{Sol:picture}. Ignore the green $X, BX$ and $NX$ as they arrive later in the game.
  
\section{Resolution and invariants}

\subsection{Resolution}
We are interested in \emph{embedded} resolution of singularities: we have a singular, say irreducible, subvariety $X \subset Y$ where $Y$ is smooth. The procedure we describe is canonical enough so that one can resolve any variety using local embeddings, but I will not discuss this reduction.

An embedded \emph{resolution of singularities} of $X\subset Y$ is a modification $\pi:Y' \to Y$ where 
\begin{itemize} 
\item $Y'$ is still smooth, 
\item $\pi$ is an isomorphism over the generic point of $X$, and 
\item the strict transform $X'$ of $X$ is smooth.
\end{itemize}

One important note about my exposition: \emph{I allow $Y'$, and hence $X'$, to be a stack.} I hope you will allow me this transgression. I will just note that, if you do want to arrive at a resolution with a variety at the end, the remaining problem is much easier --- it has to do with resolving abelian quotient singularities, for which toric methods apply. A more sophisticated approach uses \emph{Bergh's destackifiaction algorithm.} 

\subsection{Invariants}

\begin{itemize} 
\item 
By a \emph{singularity invariant} we mean a rule\footnote{Really a functor, see below.} that, given $X \subset Y$ gives a function $\inv_X: |X| \to \Gamma$, where
\begin{itemize}
\item $\Gamma$ is a well-ordered set,
\item $\inv_X$ is upper-semicontinuous, and
\item $\inv_X(p) = \min\Gamma\quad  \Leftrightarrow \quad p\in X$ is nonsingular.
\end{itemize}
\item
The invariant is \emph{functorial} if whenever $\phi: Y_1 \to Y$ is smooth, with $X_1 = \phi^{-1} X$, we have
$$\inv_{X_1}(p_1) = \inv_X(\phi(p_1)).$$
\item 
The invariant is \emph{smooth} if the maximal locus 
$$\{p\in X \ : \ \inv_X(p) \text{ is maximal }\}$$
in $X$ is smooth.
\end{itemize}
We note that a maximal value exists since upper-semicontinuity implies that $\inv$ is constructible, so by noetherianity of $X$ a maximum exists. 

Also note that upper-semicontinuity implies that the maximal locus is closed.

\section{Resolution in characteristic 0 --- why does it work?}
I  now explain \emph{why} resolution of singularities in characteristic 0 works, following \cite{Wlodarczyk-cobordant}.
\subsection{Statement} For a center $\oJ$ on $Y$ we defined in Section \ref{Sec:weighted} the deformation to the weighted normal cone $B$. 

If $X\subset Y$ is a subvariety, we denote by $BX\subset B$ the proper transform of $X$: it is the closure in $B$ of $X \times \GG_m$ which sits inside $\{s \neq 0\} = Y \times \GG_m \subset B$. Its intersection with $\{s=0\}$ is the normal cone $N_{Z\subset X}$ of $Z$ in $X$ (that is, the weighted normal cone). 

Similarly, the \emph{weighted blowup} of $X$ is $Bl_{\oJ}(X) := [(BX\setminus V)/\GG_m] \subset Bl_{\oJ}(Y).$ It is the proper transform of $X$ in $Bl_\oJ(Y)$. The picture can be viewed in your solution of Exercise \ref{Ex:picture} or in Section \ref{Sol:picture} (and now the green items $X, BX, NX$ should make sense).

Now assume we have a smooth, functorial singularity invariant $\inv$.
\begin{lemma}\label{Lem:criterion}
 \begin{enumerate} 
\item Fix $X \subset Y$ and assume we have a center $$\oJ = (x_1^{1/w_1},\ldots, x_k^{1/w_k})$$ on $Y$ which satisfies
\begin{enumerate}
\item $V(\oJ)$ is the maximal locus of $\inv_X$ on $X$, and
\item $\max \inv_X = \max \inv_{BX}$.
\end{enumerate}
Then $$\max\inv_{Bl_\oJ(X)} <  \max \inv_X.$$
\item Assume a center $\oJ$ satisfying these conditions exists for every resolution situation $X \subset Y$. Then resolution of singularities holds.
\end{enumerate}
\end{lemma}  

\begin{exercise} Try to prove the lemma without looking below.\end{exercise}

\subsection{Proof of part 2}
 We have $\max\inv_{Bl_\oJ(X)} <  \max \inv_X$, and by assumption $X' := Bl_\oJ(X)$ has its own center and blowup with the same property, giving rise to a sequence of birational morphisms $X^{(k)} \to \cdots \to X' \to X$ where $\max\inv_{X^{(k)}} <  \max \inv_{X^{(k-1)}}.$ Since $\Gamma$ is well ordered, this sequence of maximal values terminates, and the process only terminates if $X^{(k)}$ is smooth.

\subsection{The maximal locus} Before we  show part (1), a claim. You want to look at the big picture of $B$ you drew in Exercise \ref{Ex:picture}.

\begin{claim} The maximal locus $W$ of $\inv_{BX}$ is $V$. 
\end{claim}
\begin{proof}[Proof of claim]

By smoothness of the invariant, $W$ is smooth.

First we have by assumption (a) that the maximal locus of $\inv_X$ is $Z$. 

Since the locus in $BX$ where $s\neq 0$ is just $X\times \GG_m$, and $\GG_m$ is smooth, we have by functoriality that the invariant in this locus coincides with that on the image in $X$. So the maximal locus of $\inv_{BX}$ on the locus $s\neq 0$ is just $Z \times \GG_m$.

 By assumption (b) the maximal value on this locus coincides with the maximal value on the whole of $BX$, which is attained along $W$. Since the maximal locus $W \subset BX$ is closed, it contains $V$ as an irreducible component. Since $V$ and $W$ are smooth, \begin{center}$V$ is a connected component of $W$.\end{center}

Suppose there is a point $b\in N_{Z\subset X} \smallsetminus Z$ where the maximal value is also attained. Again by functoriality, the invariant attains its maximum along the whole orbit $C= \{t\cdot b\,:\, t\in \GG_m\}$. Since the maximal locus is closed, the same is true of the closure $\bar C$ of $C$, which meets $Z$ at a point $b_0$. In other words, $V$ is not a connected component of $W$, contradicting the highlighted conclusion above. 

Thus the claim holds, and the maximal locus of $\inv_{BX}$ is $V$.
\end{proof}
\subsection{Proof of part (1).}
Continuing with part (1) of the lemma, the claim
means that  $$\max_{b\notin V} \inv_{BX} < \max \inv_{BX} = \max \inv_X,$$
where the last equality is assumption (2) of the Lemma used once more.

Note that $B\setminus V \to Bl_\oJ(Y)$ is a smooth morphism, being the quotient by a smooth group-scheme. The preimage of $X' = Bl_\oJ(X)$ is precisely $BX \smallsetminus V$. Using functoriality again, we have that 
$$\max \inv_{Bl_\oJ(X)} = \max_{b\notin V} \inv_{BX}.$$
Combining with the inequality above we obtain 

$$\max \inv_{Bl_\oJ(X)} < \max \inv_X,$$ as required.
\hfill\qed

\section{Resolution in characteristic 0 --- how does it work?}
\subsection{Statement}
From here on we assume our base field has characteristic 0.
Here is our main task:
\begin{theorem}[{\cite{ATW-weighted}, see also \cite{McQuillan}}]\label{Th:resolution} There is a smooth, functorial singularity invariant $\inv$ and, for every $X\subset Y$, there is a center $\oJ$ satisfying (a) and (b) of Lemma \ref{Lem:criterion}(1). 
\end{theorem}
So by Lemma \ref{Lem:criterion}(2) resolution of singularities in characteristic 0 holds true.

The main purpose of this section is therefore to explain \emph{how} resolution of singularities in characteristic 0 works. 

\subsection{Preview examples}

Before we dive into a series of definitions, examples will show the invariant and center in action.
\begin{example} Let $X = V(x^2-y^2z)$. The most singular point is the origin $V(x,y,z)$, where the variables $x,y,z$ appear in degrees $2,3,3$ respectively. So the invariant we will define has value $(2,3,3)$ at the origin, one writes a corresponding center $J= (x^2,y^3,z^3)$, with associated \emph{reduced} center $\oJ = (x^{1/3},y^{1/2},z^{1/2}) = J^{1/6}$ which we can actually blow up.

On the blowup we have $x = s^3x', y=s^2y', z=s^2z'$. Plugging in the equation becomes $s^6(x'^2+ y'^2 z')$ on $B$. The proper transform $BX$ is given by $V(x'^2+ y'^2 z')$, and indeed its maximal locus is again the locus $V(x',y',z')$. This locus is removed in $B_+$, hence the invariant drops.

\begin{exercise}\label{Ex:Whitney-Newton} Draw the Newton polygon of $X = V(x^2-y^2z)$ and of the center $(x^2, y^3, z^3)$. What can you surmise on their relationship? (See Section \ref{Sol:Whitney-Newton}.)\end{exercise}
\begin{exercise}\label{Ex:Whitney} Do your best at drawing a picture of the blowup above, in a way that shows how it resolves the pinch point.  (See Section \ref{Sol:Whitney}.)\end{exercise}

\end{example}

\begin{example} Consider the same $X$ but with the blowup but with $\oJ=(x,y,z)$. In this case the proper transform $BX$ is defined by the equation $$x'^2 + y'^2z's.$$ we notice that at the point where $x'=y'=z'=s=0$ has these variables appearing in degrees $2, 4,4,4$ which is larger than the original $2,3,3$. In other words, it does not satisfy the criterion. In this case, as is well-known, the standard blowup of $X$ has an isomorphic singularity: for any $z'\neq 0$ the equation above is isomorphic to $X$, so the invariant stays the same after blowup.
\end{example}

\begin{exercise}\label{Ex:Newton} Draw the Newton polygon of $X = V(x^5+x^3y^3+ y^{100})$.  (See Section \ref{Sol:Newton}.)\end{exercise}

\begin{example}[Sketch] $X = V(x^5+x^3y^3+ y^{100})$. Considering the Newton polygon,  $\inv(0)=(5,7.5), J= (x^5,y^{7.5}), \oJ = (x^{1/3}, y^{1/2})$. The proper transform is $x'^5+x'^3y'^3+ y'^{100}s^{185}$ and the criterion holds. The invariant of the blowup is $(3,185)< (5,7.5)$ lexicographically.
\end{example}

\begin{exercise} Carry out the example in detail.\end{exercise}

\subsection{Derivatives}
Consider an ideal $\cI \subset \cO_Y$. We write $\cD_X^{\leq a}$ for the sheaf of differential operators of order at most $a$.
\begin{definition} $\cD^a(\cI) = \{\nabla f \ : \ f\in \cI,  \nabla \in \cD_X^{\leq a}\}$. 
\end{definition}
In characteristic 0 we have $\cD^a(\cI) = \cD^1\cD^{(a-1)}(\cI)$. 
\begin{example}
$\cD((x^2+y^2z)) = (x,yz,y^2)$. 
\end{example}
\begin{exercise} Compute $\cD(x^5+x^3y^3+ y^{100})$.\end{exercise}
\subsection{Order}
\begin{definition} $\ord_\cI(p) \ =\  \min\{\,a\ :\ (\cD^a(\cI))_p = \cO_{Y,p}\,\}$ (and the order of the zero ideal is $\infty$).
\end{definition}
Note: $\ord_\cI(p) \geq a \quad \Leftrightarrow \quad p \in V(\cD^{a-1}(\cI)).$

This implies: 
\begin{observation} The function $\ord_\cI$ is upper semicontinuous. \end{observation}

Note also that 
\begin{observation}  The function $\ord_\cI$ is functorial for smooth morphisms. 
\end{observation}
This follows because derivatives are compatible with \'etale morphisms; up to \'etale morphisms a smooth morphism is like a product with affine space; and $\ord_\cI$ is compatible with a product with affine space by direct computation.

\begin{exercise} Find the order of $x^2 + y^2z$ and of $x^5 + x^3y^3 + y^{100}$ at the origin.
\end{exercise}
\subsection{Maximal contact}
\begin{definition}[{\cite{Giraud}}] Consider a nonzero  ideal $\cI$ with $\ord_\cI(p) = a>0$. An element $x \in \cO_p$ is a \emph{maximal contact element}, and its vanishing locus $H = V(x)$ is a \emph{maximal contact hypersurface at $p$}, if $x \in \cD^{a-1}(\cI)_p$ and $\cD^1(x) = \cO_p$. 
\end{definition}
\begin{lemma} In characteristic 0, a maximal contact element  exists.
\end{lemma}
Indeed, $\cO_p = \cD^a(\cI) = \cD^1(\cD^{a-1}(\cI))$ so there is $x\in \cD^{a-1}(\cI)$ with $\cD^1(x) = \cO_p$.

Note that $x \in \cD^{a-1}(\cI)$ implies that $$V(x) \supset V(\cD^{a-1}(\cI)) = \{p\ : \ \ord_\cI(p) \geq a\}.$$
So: 
\begin{observation} A maximal contact hypersurface is a smooth hypersurface which \emph{schematically} includes all nearby points of order $a$.\end{observation}
\begin{remark} A maximal contact element does not necessarily exist in characteristic $p$: R. Narasimhan \cite{Narasimhan} gave the example of a hyersurface $$V(x^2 + yz^3 + zw^3 + y^7w)\subset \AA^4$$ in characteristic 2 with a curve of singular points having 4-dimensional tangent space at the origin. You can still resolve it, see \cite{Hauser-blowups}. Hironaka's replacement for maximal contact ($x=0$ in this case)  is a hypersurface that realizes the characteristic polyhedron, but I digress.
\end{remark} 

\begin{exercise} Find at least two maximal contact elements for each of $x^2 + y^2z$ and of $x^5 + x^3y^3 + y^{100}$ at the origin.
\end{exercise}
A maximal contact hypersurface is thus a candidate for induction on dimension. But what subscheme should we consider on it?

\subsection{Coefficient ideals}

Consider an ideal $\cI$ and an integer $a$ (which usually stands for $a = \ord_\cI(p)$). We consider the ideals $\cI, \cD^1(\cI),\ldots,\cD^{a-1}(\cI)$ with corresponding weights (standing for expected orders at $p$) $a, a-1,\ldots,1$. 
\begin{definition}[{\cite{Kollar}}] The coefficient ideal of $\cI$ at $p$ is
$$C(\cI,a):= \sum_{\sum ib_i\geq a!} \prod _{i=0}^{a-1} (\cD^{a-i}(\cI))^{b_i}.$$
\end{definition}
\begin{remark}
Notice that if $\ord_\cI(p) = a$ then $\ord_{C(\cI,p)}(p) = a!$, so there is a rescaling factor of $(a-1)!$ to be taken into account. The canonical objects to be considered are $\QQ$-ideals, in which case one considers  $\cI$ which has order $a$ along with $\cD^{a-i}(\cI)^{a/i}$ also with order $a$, and then $$C(\cI,a)^{1/(a-1)!} = \sum_{i=1}^a \cD^{a-i}(\cI)^{a/i}$$ also has order $a$. We will not use this canonical formalism here. It is used in a paper of  Bierstone  and Milman \cite[Section 3.3]{Bierstone-Milman-functorial}, in Wlodarczyk's treatment \cite{Wlodarczyk-cobordant},  and in our foliations paper \cite{ABTW-foliated}.
\end{remark}
The coefficient ideal takes into account all possible terms that should contribute to the restriction of $\cI$ to any maximal contact hypersurface. 
If $x$ is a maximal contact element, we write $\cI[x] := C(\cI,a)|_{V(x)}$. 

\begin{exercise} Let $\cI = (f)$ be principal, and $f = x^a + c_1x^{a-1} + c_2x^{a-2} + \cdots + c_a.$ Show that when $c_1 = 0 $ the element $x$ is a maximal contact. In this case show that $c_i \in \cD^{a-i}(\cI)|_{V(x)}$, and that a suitable power contributes to $\cI[x]$.\end{exercise}

\begin{exercise} Compute $C(\cI,a)$ and $\cI[x]$ in the examples  $x^2 + y^2z$ and  $x^5 + x^3y^3 + y^{100}$ at the origin.
\end{exercise}

\subsection{Invariant and center}

We can now define both an invariant and a center inductively, assuming an invariant $\inv_{\cI}(p)$ and center $J_\cI$ were defined in dimension $\dim (Y)-1$. 
\begin{definition}
Suppose $\ord_\cI(p) = a_1$ with maximal contact element $x_1$. Suppose $$\inv_{\cI[x_1]} =  (b_2,\ldots,b_k)$$ with center $$J_{\cI[x_1]} = (\ox_2^{b_2},\ldots, \ox_k^{b_k}).$$
set $$a_i = b_i/(a_1-1)!$$ and let $x_2,\ldots,x_k$ be arbitrary elements restricting to $\ox_i$ on $V(x_1)$.
Define 
$$\inv_\cI(p) := (a_1,a_2,\ldots,a_k)$$
and $$J_\cI := (x_1^{a_1},x_2^{a_2},\ldots,x_k^{a_k}).$$
\end{definition}

\begin{exercise} Compute the invariant and center in the examples  $x^2 + y^2z$ and of $x^5 + x^3y^3 + y^{100}$ at the origin.
\end{exercise}

\begin{exercise} What should one do if $\cI$ is a unit ideal? What should one do if $\cI[x_1] = 0$?
\end{exercise}

We order invariants lexicographically, where a truncation of a sequence is larger: $(2)>(2,3)$. The set of values it takes is easily shown to be well-ordered --- while the entries are rational, their denominators are bounded in terms of the earlier entries.

\begin{exercise} Verify this last statement.
\end{exercise}

\begin{remark}
There is one property of the invariant and center I do not address here --- namely that they are well-defined. The only issue is the choice of a maximal contact element $x_1$ and the lifts $x_i$ of $\ox_i$. 
Proposition \ref{Prop:center-unique} below gives a reason why it is well defined.
\end{remark}
\subsection{Properties of the invariant}

For ideals, ``resolution" is replaced by a sequence of weighted blowups so that the proper transform (or more precisely the controlled transform) is the unit ideal --- this is called ``principalization". There is a reduction of embedded resolution to this task, which  I do not discuss here. Also, Lemma \ref{Lem:criterion} has an evident analogue for ideals, for which the following is relevant:

\begin{proposition} The function $\inv_\cI$ is a smooth functorial singularity invariant.
\end{proposition}
\begin{proof} 

{\sc  The invariant $\inv_{\cI}$ is a ``singularity invariant" for ideals.}
 
We have already noted that the value set of the invariant is well-ordered.

We show $\inv_\cI $ is upper semicontinuous. Note that $a_1 = \ord_\cI$ is upper semicontinuous, and its maximum locus is contained in $V(x_1)$. By induction $\inv_{\cI[x_1]}$ is upper semicontinuous on $V(x_1)$. It follows that  $\inv_\cI $ is upper semicontinuous.

The invariant's minimum value is $(0)$ taken only for the unit ideal. For resolution of singularities, the invariant of the ideal of a  subvariety of codimension $d$ takes its minimum value $(1,\ldots,1)$, with $d$ entries, at smooth points of $X$. 

{\sc  The invariant $\inv_{\cI}$ is functorial.} We have seen that $\ord_\cI$ is functorial for smooth morphisms. Let us now choose $x_1$ --- it remains maximal contact when pulled back by a smooth morphism. By induction $\inv_{\cI[x_1]}$ is functorial for smooth morphisms, hence  
 $\inv_{\cI}$ is functorial.
 
{\sc  The maximal locus is smooth.} The maximal locus of $\ord_\cI$ lies inside $V(x_1)$. Inductively the maximal locus of $\inv_{\cI[x_1]}$ is $V(x_2,\ldots,x_k)$. Hence the maximal locus of $\inv_\cI$ is    $V(x_1,x_2,\ldots,x_k)$, which is smooth.
\end{proof}

\subsection{Admissibility}  A center $J$ is admissible for an ideal $\cI$ if $v_J(f) \geq 1$ for all $f \in \cI$. 
\begin{lemma} The center $J_\cI$ is admissible for $\cI$. More generally, if $f\in \cD^1(\cI)$ then $v_J(f) \geq 1-i/a_1$, and if $f \in C(\cI,a_1)$ then $v_J(f) \geq (a_1-1)!$\end{lemma}

To see this, expand  an element $f\in \cI$ as $$f= c_{0} + c_1 x_1  + \cdots +  c_{a_1-1} x_1^{a_1-1} +  c_{a_1}x_1^{a_1} +\cdots,$$ where  $$c_i = \frac{1}{i!}\left(\frac{\partial^if}{\partial x_1^i}\right )_{x_1 = 0}\  \in\  \cD^{i}(\cI)|_{V(x_1)},$$ and use induction again. 

The following is the real reason the center is well-defined:

\begin{proposition}\label{Prop:center-unique} $J$ is the unique center of maximal invariant which is admissible for $\cI$.
\end{proposition}

I will not go into the proof of this, though it is not difficult.

\begin{remark} One may define an invariant in positive characteristic in an analogous manner, by taking the unique  of maximal invariant which is admissible for $\cI$.
This leads to an invariant which does not satisfy  \ref{Lem:criterion}(a) --- this fails for Whitney's umbrella in characteristic 2 --- neither it is a smooth invariant, as one can see from the system $(x^2 + y^2z, z^{100})$. These issues are again raised in \cite{Wlodarczyk-cobordant}.
\end{remark}
\subsection{Reduced center and transform}

The center $J_\cI := (x_1^{a_1},x_2^{a_2},\ldots,x_k^{a_k})$ is not in the form we used for weighted blowups. There is, however, a positive integer $\ell$ such that $$\oJ_\cI := J_\cI^{1/\ell} = (x_1^{1/w_1},\ldots,x_k^{1/w_k}),$$ with $w_i$ positive integers. We use this rescaled center $\oJ_\cI$ to prove Theorem \ref{Th:resolution}. 

We note that the associated monomial valuations satisfy $v_\oJ = \ell \cdot v_J$. 


We have already shown  condition \ref{Lem:criterion}(a) above. In order to consider the remaining condition \ref{Lem:criterion}(b), we need to look at the transform of $\cI$.

Recall the change of variables $x_i = s^{w_i} x_i'$ for $i=1,\ldots,k$.  Write $\sigma: B \to Y$.

We have that if $f(x_1,\ldots,x_n) \in \cI$ then $\sigma^*f = f(s^{w_1} x_1',\ldots, s^{w_k} x_k', x_{k+1}\ldots,x_k)$. It follows from admissibility that $\sigma^*f = s^\ell f'$ for some regular function $f'$ --- use that $a_i w_i = \ell$. 

If instead $f \in \cD^i(\cI)$ then we can always write $\sigma^*f = s^{\ell - iw_1} f'$ for some regular function $f'$ --- we see why below. This implies a similar result for the coefficient ideal, where $\ell$ is replaced by $(a-1)!\cdot\ell$.

\subsection{Transforms of derivatives and coefficient ideals}
For an ideal $\cI$ denote by $\cI'$ its controlled transform $\{f': f\in \cI\}$.

\begin{lemma}[See {\cite[Lemma 4.5]{Encinas-Villamayor-good},
\cite[Lemma 2.6.2]{Wlodarczyk}, 
\cite[Lemma 3.3]{Bierstone-Milman-funct}}] \label{Lem:transform-derivative} $(\cD(\cI))' \subset \cD(\cI')$.
\end{lemma}
\begin{proof} It suffices to show that $(\cD(\cI))' \subset \cD_{B/\AA^1}(\cI')$.

As indicated above, $\cI' = (s^{-\ell} f(s^{w_j}x_j) \ :\ f\in \cI\,)$.

It follows that $\cD_{B/\AA^1}(\cI') = \left(\frac{\partial}{\partial x'_i}(s^{-\ell} f(s^{w_j} x_j')) \mid f\in \cI, i=1,\ldots,n\right) $, and applying the chain rule this equals $ \left(s^{-\ell} \frac{\partial f}{\partial x_i} (s^{w_j} x_j') \cdot s^{w_i}\right)$, giving 
$${\cD_{B/\AA^1}(\cI') = \left(s^{-\ell+w_i} \frac{\partial f}{\partial x_i} (s^{w_j} x_j'))\right)}.$$ 

On the other hand $\cD(\cI) = \left(\frac{\partial f}{\partial x_i}(x_i)\right)$ giving
$$\cD(\cI)' = \left(s^{-\ell+w_1}\frac{\partial f}{\partial x_i}(s^{w_j}x_j')\right).$$
Note that $w_1\geq w_i$ for all $i$, so
		$$\cD(\cI)' = \left(s^{-\ell+w_1}\frac{\partial f}{\partial x_i}(s^{w_j}x_j')\right) \subset \left(s^{-\ell+w_i} \frac{\partial f}{\partial x_i} (s^{w_j} x_j'))\right)=\cD_{B/\AA^1}(\cI'),$$
		as  needed.

\end{proof}


\begin{proposition}\label{Prop:transform-coefficient}  $\cI'^{(a_1-1)!} \subset C(\cI,a_1)' \subset C(\cI',a_1).$\end{proposition}

\begin{proof}
 The inclusion $\cI^{(a_1-1)!} \subset C(\cI,a_1)$ implies an inclusion of transforms $\cI'^{(a_1-1)!} \subset C(\cI,a_1)' $.
 
 Lemma \ref{Lem:transform-derivative} above implies the second inclusion.
\end{proof} 

We note that the invariant and center of $\cI'^{(a_1-1)!}$ equals that of  $C(\cI',a_1).$ Proposition \ref{Prop:transform-coefficient}  says that this coincides with the invariant of $C(\cI,a_1)'$. In other words: to check the resolution criterion for $\cI$ it suffices to check it for $C(\cI,a_1).$

This allows us to conflate $\cI$ with its coefficient ideal. I'll ignore the rescaling factor that this requires. After all, I am telling you \emph{how} it works, the complete proofs are detailed elsewhere.

\subsection{Order and  maximal contact element on $B$}

\begin{claim} $\max_B \ord_{\cI'} = a_1$ and wherever this is attained, $x_1'$ is a maximal contact element.\end{claim}

For this discussion we assume $\cI$ is already a coefficient ideal, so in particular $\partial_{x_1}^{a_1}(\cI) = \cO$ and $x_1 \in \partial_{x_1}^{a_1-1}(\cI)$. This is harmless by the previous proposition. (It can also be arranged by a change of variables.)

This means that there is an element $f\in \cI$ such that  $x_1 = \partial_{x_1}^{a_1-1}(f)$ and so $\partial_{x_1}^{a_1}(f)=1$.

Plugging in we obtain that $\partial_{x_1'}^{a_1}(s^{-\ell} f(s^{w_i} x_i'))=1$, since $a_1w_1 = \ell$. This means that the order everywhere is at most $a_1$. Where it is, we indeed get  $x_1 = s^{-\ell+w_1} \partial_{x_1'}^{a_1-1}(f(s^{w_i} x_i'))$ so $x_1' = \partial_{x_1'}^{a_1-1}(s^{-\ell}f(s^{w_i} x_i'))$. So it is a maximal contact element.

\subsection{The invariant of $\cI'$} 
\begin{claim} $\max_B \inv_{\cI'} = (a_1,\ldots,a_k)$\end{claim}
We note that this implies  \ref{Lem:criterion}(b), hence resolution of singularities.

As we have seen, the maximal order is $a_1$ and it is attained within the hypersurface $H' = \{x_1'=0\}$, since it is maximal contact. Note that $H'$ is the degeneration to the normal cone of $H = \{x_1 = 0\}$, with center $\oJ_H = (x_2^{1/w_2},\ldots,x_k^{1/w_k})$.

We can now apply induction, so that, ignoring annoying rescaling factors, $(a_2,\ldots,a_k)$ is the maximal invariant of $\cI'[x_1']$ with center $J_{H'}=(x_2'^{a_2},\ldots,x_k'^{a_k})$, so by definition the maximal invariant if $\cI'$ is indeed $(a_1,\ldots,a_k)$, with center $(x_1'^{a_1},\ldots,x_k'^{a_k})$.
 
We are done!

\section{Enter foliations}
\subsection{Resolution of foliations}
A \emph{foliation} is a coherent subsheaf $\cF \subset \cD_X$ which is closed under Lie brackets. We also consider \emph{logarithmic foliations}   $\cF \subset \cD^{\log}_X$, for instance when $X$ carries a normal crossings divisor.
A foliation is \emph{smooth} if it is a subbundle of $\cD_X$, and logarithmically smooth if it is a subbundle of $\cD^{\log}_X$

One major question is that of resolution of singularities of foliations, sometimes called ``reduction of singularities" since smoothness is rarely achievable. There are contributions in the CIRM workshop which   tell us about the history (I need to study it myself). Some important names: Seidenberg \cite{Seidenberg}, Aroca \cite{Aroca-reduction}, Panazzolo \cite{Pan}, McQuillan--Panazzolo  \cite{McQPan}.

A  paper we wrote \cite{ATW-relative} is devoted to showing relative logarithmic resolution, which implies in particular log resolution of a foliation of the form $\cD_{X/B}$. 

Log smoothness after blowups is impossible already in 2-dimensional examples, as one sees with saddle-nodes.

\subsection{Resolution inside foliated manifolds}
All results I am involved in from here on are in \cite{ABTW-foliated}.

We focus on embedded resolution: say $(Y,\cF)$ is a foliated manifold and $X\subset Y$ is a subvariety. What can we do to resolve $X$ in a way that respects $\cF$?

There are two opposing cases of interest:
\begin{theorem}\label{Th:invariant}  Assume that $X$ is $\cF$-invariant, namely $\cF(\cI_X) \subset \cI_X$. Then the functorial resolution of $X$ described in previous sections is $\cF$-invariant: the foliation $\cF$ lifts in the resolution process to a foliation $\cF'$ on $Y'$, and $X'$ is $\cF'$-invariant. \end{theorem}
Indeed, in this case one can think of generators of $\cF$ as infinitesimal automorphisms of $Y$ preserving $X$, and by functoriality of resolution (in the formal or analytic setting) in is equivariant for such group actions. This was used by people using Hironaka's resolution, for instance Kawamata \cite[Lemma 3.6]{Kawamata-fiber-spaces}.

The following seems new:

\begin{theorem}\label{Th:transverse} Assume that $X$ is generically transverse to $\cF$. Then there is a sequence of $\cF$-aligned blowups, with resulting $X' \subset Y', \cF'$, such that $X'$ is transverse to $\cF'$.
\end{theorem}

This brings up a question: what is an aligned blowup? Again there are two extreme cases: if the center $J$ is $\cF$-invariant then it is $\cF$-aligned. But also important is the case when  the $x_i$ are transverse of $\cF$. An $\cF$-aligned center is a mixture of the two: on formal completion it can be written as $(x_1^{a_1},\ldots, x_k^{a_k}, \ y_1^{b_1}, \ldots, y_k^{b_l})$, where $x_i$ are transverse and the partial center $( y_1^{b_1}, \ldots, y_k^{b_l})$ is invariant.

The two theorems above are  consequences of the following:

 \begin{theorem} For $X\subset Y, \cF$ as above, there is a functorial sequence of $\cF$-aligned blowups, arriving at $X'\subset Y', \cF'$, so that $X'$ is smooth and $\cI_X$ is an $\cF'$-aligned center itself.\end{theorem}

 How do we prove this? We construct an invariant and a center, and show that the criterion \ref{Lem:criterion} applies. One starts with defining the $\cF$-order, the analogue of order but only using derivatives from $\cF$. If finite, one defines $\cF$-maximal contact and $\cF$-coefficient ideals as before. If the $\cF$-order is infinite, namely $\cI$ is $\cF$-invariant, one reverts to usual derivatives, but with $\cF$-invariant being $\infty+\inv_\cI$, see examples below. At the end this all works nicely.
 
 We compute transforms of ideals and of foliations on $B_+$ rather than on $X'$. Specifically for foliations, we work with $\cD^{\log}_{B_+/\AA^1}$, which provides descent to $\cD^{\log}_{X'/\cA^1}= \cD^{\log}_{X'}$, where $\cA = [\AA^1/\GG^m]$. One point to be careful about: even when $\cF' \subset \cD_{B_+/\AA^1}$ is smooth, the corresponding foliation on $X'$ may be only logarithmically smooth.
 
 \subsection{Examples of foliated principalization}
 
  \begin{example}[For Theorem \ref{Th:transverse}]  Let $X = V(y^3-x^2)$ and $\cF= \Span\left(\partial_y\right)$. We define the $\cF$-order in an analogous way to the order, using only $\cF$. Then the $\cF$-order at the origin is 3 with $\cF$-maximal contact $y$. We have $\cI[y]= (x^4)$ with rescaling factor $(3-1)! = 2$. Its $\cF$-oder is infinite, and its $\cD$-order is 4, which rescales down to $4/2=2$.  We therefore assign $\cF$-invariant $(3,\infty+2)$ with $\cF$-center $(y^3,x^2)$. The blowup has equation $(y'^3-x'^2)$ which is transverse to the transformed foliation $\cF'=\Span(\partial_{y'})$ as we have removed $x'=y'=0$.
 \end{example}

 \begin{example}[For Theorem \ref{Th:invariant}] Let $X = V(y^3-x^2)$ and $\cF= \Span\left(2y\partial_y+3x\partial_x\right)$. Since $X$ is $\cF$-invariant the $\cF$-order is $\infty$. As the usual invariant is $(2,3)$ with center $(x^2,y^3)$ we assign $\cF$-invariant $(\infty+2,\infty+3)$ still with center $(x^2,y^3)$.  The blowup $X'$ is regular, has equation $(y'^3-x'^2)$ with $x'=y'=0$ removed, and is invariant under the transformed foliation $\cF'= \Span\left(2y'\partial_{y'}+3x'\partial_{x'}\right)$.
 \end{example}

\begin{exercise} The examples are computed on $B_+$. Compute the transforms on \'etale charts for the blow-up. \end{exercise}

 \subsection{Preservation and reduction in good foliation classes}
 Let us look back at the problem of reducing the singularities of a foliation.
 \begin{theorem} The transform under an $\cF$-aligned blowup of a log smooth foliation $\cF$ is log smooth. The restriction of a log smooth foliation $\cF$ to the support of an $\cF$-aligned center not contained in the logarithmic boundary  is log smooth. \end{theorem}
 
 There are further classes of foliations that are preserved under aligned blowups, for instance Belotto's class of \emph{monomial foliations}. 
 
 \begin{theorem} The transform under an $\cF$-aligned blowup of a monomial foliation $\cF$ is monomial.  The restriction to the support of an $\cF$-aligned center  of a monomial foliation $\cF$ is monomial. \end{theorem}
 
 This has the following implication: a class $\cC$  of saturated foliations that  is stable under aligned blowups, log smooth products, and  restrictions as above is said to be \emph{thick}. 
 
 A $\cC$-reduction of a foliation $\cF$ is a modification $X' \to X$ such that the saturated  pullback $\cF'$ is in $\cC$.
 
 \begin{theorem} Say $\cF$ is the saturated preimage under a dominant rational map of a foliation in a thick class $\cC$.  Then $\cF$ admits a $\cC$-reduction. In particular, 
 \begin{itemize} 
 \item A totally integrable foliation has a log smooth and monomial reduction of singularities.
 \item  A Darboux-integrable foliation has a log smooth and monomial reduction of singularities.
 \end{itemize}\end{theorem}
 
 \begin{proof}[Sketch of proof] Consider the graph embedding $X \hookrightarrow X\to B= Y$ with the foliation-pullback $\cH:= \pi_B^*\cG$. Note that $\cH$ is in $\cC$. Its restriction to $X$ is $\cF$. Consider the embedded resolution of $X$ in the foliated manifold $(Y, \cH)$. The result is $X'\subset Y', 
 \cH'$ such that $\cH'$ is  still in $\cC$ and $X$ itself is a, $\cH'$-aligned center. So the restriction $\cF'$ is in $\cC$. 
 
 Now a totally integrable foliation is the foliation-pullback of the zero foliation, which is both log smooth and monomial. A Darboux-integrable foliation is by definition the foliation pullback of a monomial foliation. One shows that a monomial foliation has a modification that is both log smooth and monomial.
 \end{proof}
 
 This makes it interesting to investigate the following:
 \begin{problem} Show that the transform under an $\cF$-aligned blowup of a log canonical foliation $\cF$ is log canonical. \end{problem}
\begin{example} Let $\cF= \Span(3y^2 \partial_x - 2x \partial_y)$. This is the relative foliation $\cD_{\AA^2/\AA^1}$ of the map $\xymatrix{\AA^2\ar[rr]^{y^3+x^2}&&  \AA^1}$, hence it is the foliation pullback of the zero foliation. As the zero foliation is log smooth, we claim that we can transform it to a log smooth foliation. 

The foliation $\cF$ is thus the restriction via the graph $(x,y) \mapsto (x,y,y^3+x^2)$ of the foliation $\cH=\Span(\partial_x,\partial_y)$ on the $xyz$-space $\AA^3$. 

The ideal of the graph is $(z-y^3-x^2)$, with $\cH\text{-}\ord_\cI=2$ at the origin. The restricted coefficient ideal is $\cI[x]=(z-y^3)$ with $\cH[x]\text{-}\ord_{\cI[x]}= 3$ at the origin and $\cH[x]$-maximal contact $y$. Finally the remaining coeficient ideal is $(z^2)$ with rescaling factor 2, having infinite $\cH[x,y]$-order and order $2$. Rescaling back, we obtain the invariant $(2,3,\infty+1)$ with center $(x^2,y^3, z)$. 

The weighted blowup has equation $(z'-y'^3-x'^2)$ on the $\GG_m$-chart $B_+$, where the origin is excluded. It is transversal to the transformed foliation $$\cH' \quad =\quad  \Span(\partial_{x'},\partial_{y'})\qquad\subset\qquad \cD_{B_+/\AA^1_s} = \cD^{\log}_{B_+/\AA^1_s}.$$ The resulting closed subscheme is the corrsponding weighted blowup of the plain. Again on the chart $B_+$  the restricted foliation $\cF'= \Span(3y'^2 \partial_{x'} - 2x' \partial_{y'})$ is nonsingular. This implies that it is logarithmically smooth on $X'$.
\end{example}

\begin{exercise} Explicitly compute the resulting foliation on \'etale charts for the weighted blowup.\end{exercise}

The examples and exercises show in no uncertain terms that the correct place to compute transforms is $B_+$.
 
\section{Suggestions for selected exercises}
\subsection{Exercise \ref{Ex:picture}}\label{Sol:picture} \hfill

\hspace*{1in}  \includegraphics[height=2.5in]{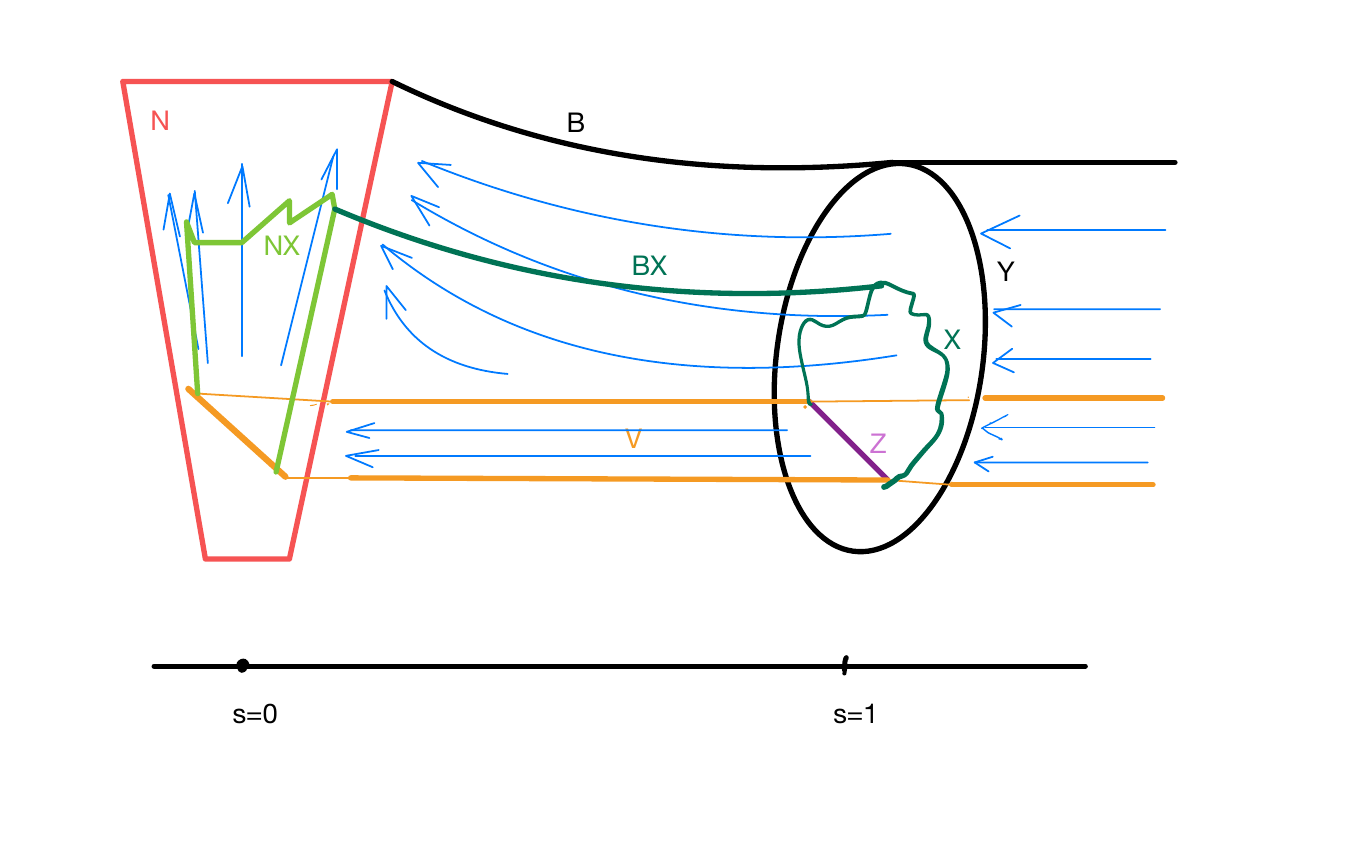}

(See \cite{ABQTW}.)

\subsection{Exercise \ref{Ex:Whitney-Newton}}\label{Sol:Whitney-Newton} The Newton polyhedron of  $x^2 - y^2z$ in pink, the center $(x^2,y^3,z^3)$ behind in blue. \hfill

\hspace*{1in}  \includegraphics[height=2in]{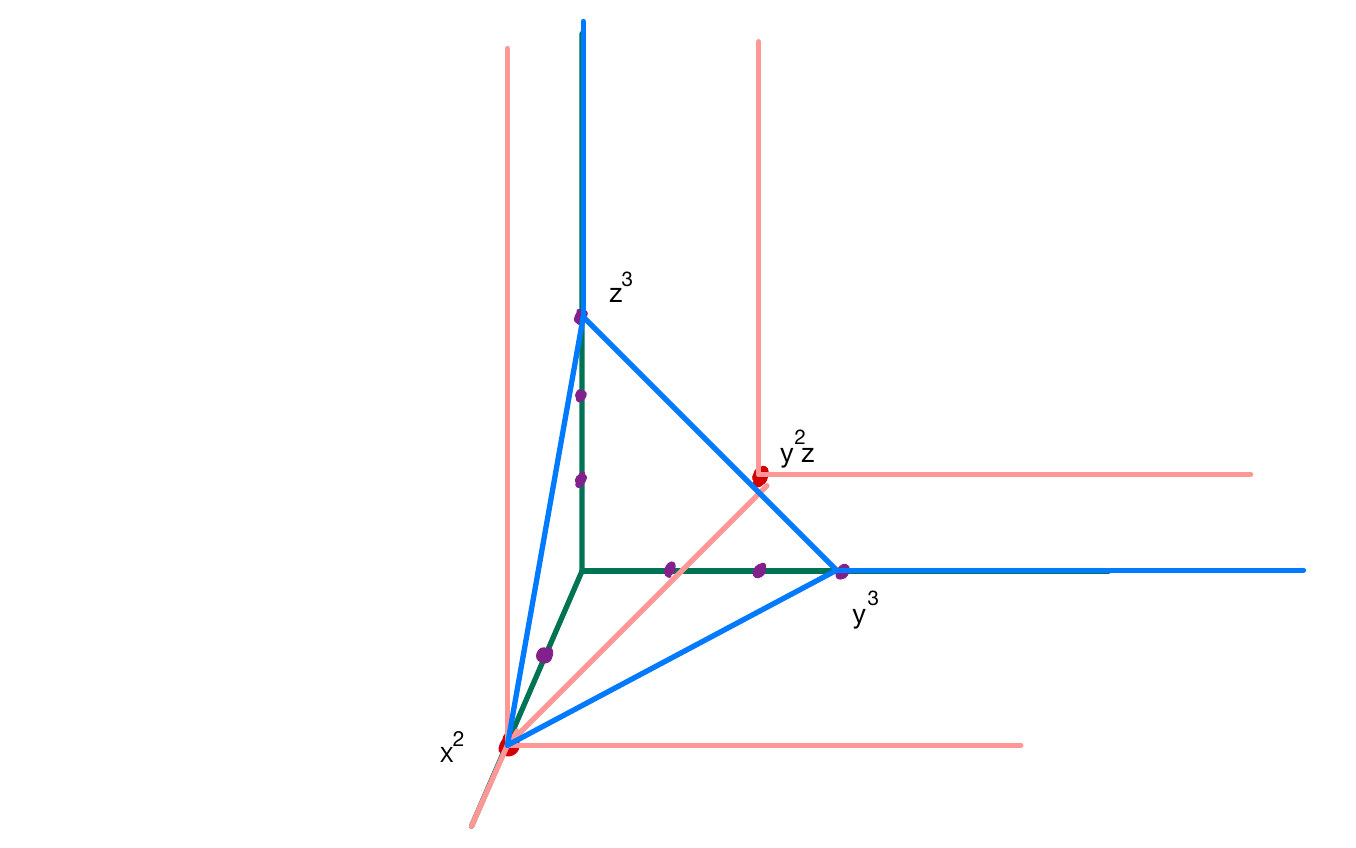}

\subsection{Exercise \ref{Ex:Whitney}}\label{Sol:Whitney} The red dot is the $\mu_3$ point $s = z'=y'=0$. The green triangle is the exceptional weighted projective plane $s=0$. The light green line is the $\mu_2$-locus $s=x'=0$. The purple curve is the intersection of the umbrella with the exceptional divisor --- the $\mu_2$ action identifies it with its mirror image, drawn with a lighter touch.  \hfill

\hspace*{1in}  \includegraphics[height=2in]{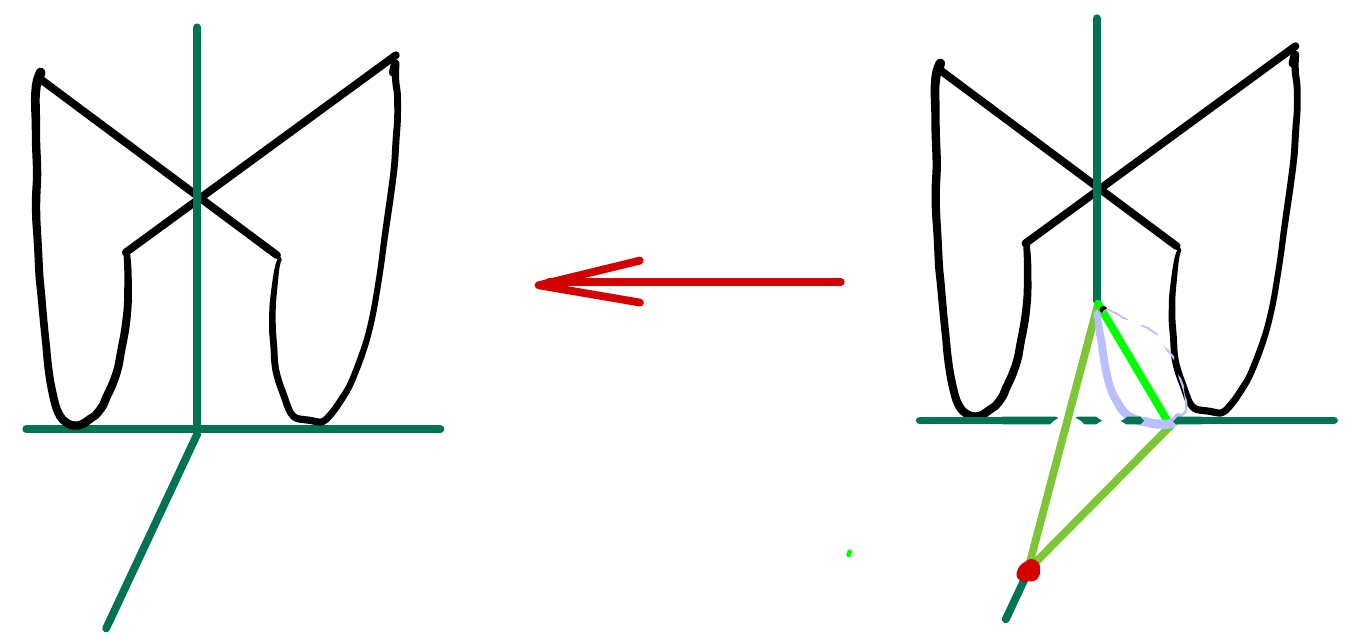}

\subsection{Exercise \ref{Ex:Newton}}\label{Sol:Newton}   \hfill

\hspace*{2in} \includegraphics[height=1.5in]{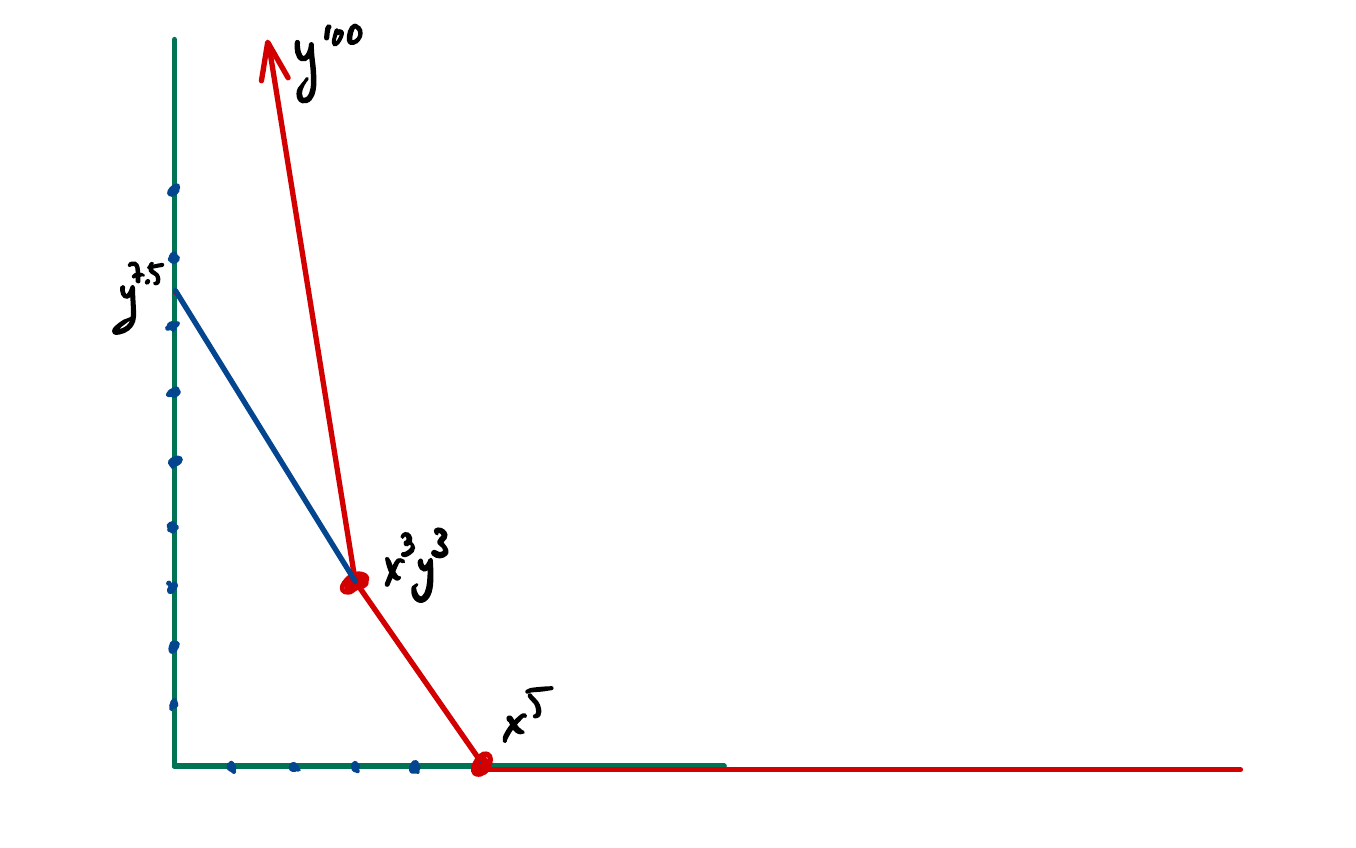}

\bibliographystyle{amsalpha}
\bibliography{principalization}

\end{document}